\documentclass{amsart}
\usepackage{fourier,wasowicz,tikz}
\usepackage{datetime}
\info{Published: Opuscula Math. 32/3 (2012), 591--600.}

\begin{document}
\author{\SW}
\address{\SWaddr}
\email{\SWmail}
\author{Alfred Witkowski}
\address{Institute of Mathematics and Physics, University of Technology and Life Sciences, al.\ prof.\ Kaliskiego~7, 85-796 Bydgoszcz, Poland}
\email{alfred.witkowski@utp.edu.pl}
\title{On some inequality of Hermite--Hadamard type}
\date{}
\begin{abstract}
 It is well--known that the left term of the classical Hermite--Hadamard ine\-quality is closer to the integral mean value than the right one. We show that in the multivariate case it is not true. Moreover, we introduce some related inequality comparing the methods of the approximate integration, which is optimal. We also present its counterpart of Fej\'er type.
\end{abstract}
\maketitle
\section{Introduction}
The Hermite--Hadamard inequality is one of the most classical inequalities in the theory of convex functions. It states that if $f:[a,b]\to\R$ is a~convex function, then
\begin{equation}\label{eq:HH}
 f\left(\frac{a+b}{2}\right)\xle\frac{1}{b-a}\int_a^b f(x)\,\text{d}x\xle\frac{f(a)+f(b)}{2}\,.
\end{equation}
This inequality is present in many textbooks and monographs devoted to convex functions and it was also extensively studied by many researchers in the past and present. Very interesting historical remarks concerning the inequality~\eqref{eq:HH} can be found in \cite{MitLac85} (see also \cite[p.~62]{NicPer06}). The following aspect of the inequality~\eqref{eq:HH} is also well--known: the inequality on the left--hand side gives better estimate of the integral mean value than the inequality on the right. It means that
\begin{equation}\label{eq:HH-better}
 \frac{1}{b-a}\int_a^b f(x)\,\text{d}x-f\left(\frac{a+b}{2}\right)\xle\frac{f(a)+f(b)}{2}-\frac{1}{b-a}\int_a^b f(x)\,\text{d}x\,,
\end{equation}
whenever $f:[a,b]\to\R$ is convex. After slight rearrangement we arrive at
\begin{equation}\label{eq:HH-better-2}
 \frac{1}{b-a}\int_a^b f(x)\,\text{d}x\xle\frac{1}{2}\left[f\left(\frac{a+b}{2}\right)+\frac{f(a)+f(b)}{2}\right]\,,
\end{equation}
which is really easy to prove, because it is enough to apply the right--hand side of~\eqref{eq:HH} to the intervals $\bigl[a,\frac{a+b}{2}\bigr]$ and $\bigl[\frac{a+b}{2},b\bigr]$ and add the obtained inequalities (\cf\ \eg\ \cite[p.~52, Remark 1.9.3]{NicPer06}). It is also easy to see by drawing a picture. However, there is also another nice geometrical proof based on the idea of symmetry. We would like to quote it here, because the mentioned above idea will be used in the paper as the main tool. To this end for a~convex function $f:[a,b]\to\R$ consider $g(x)=f(x)+f(a+b-x)$, which is also convex and, moreover, symmetric function with respect to $\frac{a+b}{2}$. See Figure~\ref{Pic1} below, where the curve symbolizes the graph of~$g$ (we silently assume nonnegativity of $g$ for the graphical purposes, but our reasoning, after slight changes, works in the general case).
\begin{figure}[ht]
 \begin{tikzpicture}[domain=0.5:3.5,>=stealth,scale=0.6]
   \draw[->] (-0.5,0)--(4,0);
   \draw[->] (0,-0.5)--(0,3.5);
   \draw[fill=gray!50] (0.5,0)--(0.5,3)--(2,0.75)--(2,0)--cycle;
   \draw (2,0.75)--(3.5,3)--(3.5,0);
   \draw plot (\x,{(\x-2)^2+0.75});
   \node[below] at (0.5,0) {\small $a$};
   \node[below] at (2,0) {$\frac{a+b}{2}$};
   \node[below] at (3.5,0) {\small $b$};
 \end{tikzpicture}
 \caption{}\label{Pic1}
\end{figure}
The integral $\int_a^b g(x)\,\text{d}x$ is not greater than double area of the trapezoid filled in gray. Hence
\[
 \int_a^b g(x)\,\text{d}x\xle\frac{b-a}{2}\left[g(a)+g\left(\frac{a+b}{2}\right)\right]\,.
\]
Taking into account $\int_a^b g(x)\,\text{d}x=2\int_a^b f(x)\,\text{d}x$ (symmetry) and the definition of~$g$, we obtain~\eqref{eq:HH-better-2}.
\par\smallskip
The multivariate counterpart of~\eqref{eq:HH} can be found in  \cite{NeuPec89,Neu90,Bes08}. In  the simpliest form it reads as follows.
\begin{thm}
 Let $T=\conv\{\,\mathbf{v}_0,\dots,\mathbf{v}_n\,\}\subset\R^n$ be the simplex with the volume $|T|$. If $f:T\to\R$ is a~convex function, then
 \begin{equation}\label{eq:HH-simplex}
  f\Biggl(\frac{1}{n+1}\sum_{i=0}^n\mathbf{v}_i\Biggr)\xle\frac{1}{|T|}\int_Tf(\mathbf{x})\,\text{\normalfont d}\mathbf{x}\xle\frac{1}{n+1}\sum_{i=0}^nf(\mathbf{v}_i)\,.
 \end{equation}
\end{thm}
The proof given in~\cite{Bes08} is straightforward and elementary, but it requires a large amo\-unt of computation. As a~byproduct of our main considerations we offer an easier proof, which is also elementary.
\par\smallskip
Let us adopt the following notation:
\begin{align*}
 \mathcal{L}(f,T)&=\frac{1}{|T|}\int_Tf(\mathbf{x})\,\text{\normalfont d}\mathbf{x}
 -f\Biggl(\frac{1}{n+1}\sum_{i=0}^n\mathbf{v}_i\Biggr)\,,\\[1ex]
 \mathcal{R}(f,T)&=\frac{1}{n+1}\sum_{i=0}^nf(\mathbf{v}_i)
 -\frac{1}{|T|}\int_Tf(\mathbf{x})\,\text{\normalfont d}\mathbf{x}\,.
\end{align*}
The natural question arises, whether the inequality similar to~\eqref{eq:HH-better} holds for convex functions of multiple variables, \ie\ if the inequality $\mathcal{L}(f,T)\xle\mathcal{R}(f,T)$ is true. This conjecture turned out to be false in the multidimensional case. To demonstrate this, consider the unit simplex $T=\conv\bigl\{(0,0),(0,1),(1,0)\bigr\}\subset\R^2$. Then
\begin{align*}
 \mathcal{L}(f,T)&=2\iint_T f(x,y)\,\text{d}x\text{d}y-f\left(\frac{1}{3},\frac{1}{3}\right)\,,\\[2ex]
 \mathcal{R}(f,T)&=\frac{f(0,0)+f(0,1)+f(1,0)}{3}-2\iint_T f(x,y)\,\text{d}x\text{d}y\,.
\end{align*}
Let $f$ be the convex function, whose graph is the surface of the pyramid shown at the Figure~\ref{Pic2} (with the lower vertex $\left(\frac{1}{3},\frac{1}{3},0\right)$).
\begin{figure}[ht]
 \begin{tikzpicture}[,>=stealth,scale=1.5]
   \draw[->] (-.5,0,0)--(1.5,0,0) node [below] {$y$};
   \draw[->] (0,-.5,0)--(0,1.5,0) node [right] {$z$};
   \draw[->] (0,0,-.5)--(0,0,1.5) node [left] {$x$};
   \draw (0,0,0)--(1,0,0)--(0,0,1)--cycle;
   \draw[dashed] (0,1,0)--(1,1,0)--(0,1,1)--cycle;
   \draw (0,0,0)--(0,1,0);
   \draw (1,0,0)--(1,1,0);
   \draw (0,0,1)--(0,1,1);
   \node[below] at (1,0,0) {\small $1$};
   \node[above] at (0,1,0) {\small $\;\;\;1$};
   \node[below] at (0,0,1) {\small $1$};
   \draw (1/3,0,1/3)--(1,1,0);
   \draw (1/3,0,1/3)--(0,1,1);
   \draw (1/3,0,1/3)--(0,1,0);
   \fill [gray,opacity=.1] (1/3,0,1/3)--(0,1,0)--(0,1,1)--cycle;
   \fill [gray,opacity=.1] (1/3,0,1/3)--(1,1,0)--(0,1,0)--cycle;
   \fill [fill=gray,nearly transparent] (1/3,0,1/3)--(0,1,1)--(1,1,0)--cycle;
 \end{tikzpicture}
 \caption{}\label{Pic2}
\end{figure}
Then $\mathcal{L}(f,T)=\frac{2}{3}>\frac{1}{3}=\mathcal{R}(f,T)$. However, there are, of course, convex functions, for which  $\mathcal{L}(f,T)\xle\mathcal{R}(f,T)$. Take, for instance, the convex function $f:T\to\R$ defined by
\[
 f(x,y)=
 \begin{cases}
  1&\text{for }(x,y)\in\{(0,0),(0,1),(1,0)\}\,,\\
  0&\text{otherwise,}
 \end{cases}
\]
for which $\mathcal{L}(f,T)=0<1=\mathcal{R}(f,T)$.
\par\smallskip
Let us stay for a while with two--dimensional case and the unit simplex~$T$. Following the idea of the proof of~\eqref{eq:HH-better-2} of dividing the interval $[a,b]$ into two parts, divide~$T$ into three subsimplices by its barycenter $\left(\frac{1}{3},\frac{1}{3}\right)$. Applying the right inequality of~\eqref{eq:HH-simplex} to each of these subsimplices and summing up the obtained inequalities, after some rearrangement we arrive at
\begin{equation*}
 2\iint_T f(x,y)\,\text{d}x\text{d}y-f\left(\frac{1}{3},\frac{1}{3}\right)
 \xle2\left[\frac{f(0,0)+f(0,1)+f(1,0)}{3}-2\iint_T f(x,y)\,\text{d}x\text{d}y\right]\,.
\end{equation*}
Therefore for two--dimensional case we have the inequality $\mathcal{L}(f,T)\xle 2\mathcal{R}(f,T)$. The above example of the convex function with the pyramidal graph shows that this inequa\-lity is optimal.
\par\smallskip
In the general case, taking into account the dimension $n=2$, we can guess that\linebreak $\mathcal{L}(f,T)\xle n\mathcal{R}(f,T)$, where $T\subset\R^n$ is a (not necessarily unit) simplex and $f:T\to\R$ is convex. As a~main result we will prove that this conjecture is true. We also prove the similar result for Hermite--Hadamard--Fej\'er inequality.

\section{Definitions and basic properties}
The convex hull of $n+1$ points $\mathbf{v}_0,\ldots,\mathbf{v}_n\in\R^n$ is called a \emph{simplex}, if the vectors $\mathbf{v}_1-\mathbf{v}_0$, $\mathbf{v}_2-\mathbf{v}_0$, \dots, $\mathbf{v}_n-\mathbf{v}_0$ are linearly independent.
The points $\mathbf{v}_i$ are called \emph{vertices} of the simplex.
 Its barycentric coordinates will be   denoted by $(\xi_0,\xi_1\,\dots,\xi_n).$
 The point
\[
 \mathbf{b}=\left(\frac{1}{n+1},\ldots,\frac{1}{n+1}\right)=\frac{\mathbf{v}_0+\cdots+\mathbf{v}_n}{n+1}
\]
is called the \emph{barycenter} of $T$.

We denote by $S_n$ the group of permutations of $n$ elements.
Any $\sigma\in S_{n+1}$ generates an affine mapping $\sigma:T\to T$ by
\[
 \sigma(\xi_0,\ldots,\xi_n)=(\xi_{\sigma(0)},\ldots,\xi_{\sigma(n)})\,.
\]
For a function $f:T\to\R$ and $\sigma\in S_{n+1}$ we define the function $f_\sigma$ by
\begin{equation}\label{eq:f_sigma}
 f_\sigma(\mathbf{x})=f\bigl(\sigma(\mathbf{x})\bigr)\,.
\end{equation}
\begin{lem}\label{lem:intf=intfsigma}
 If $f:T\to\R$ is integrable, then $\int_T f(\mathbf{x})\,\text{\normalfont d}\mathbf{x}
 =\int_T f_\sigma(\mathbf{x})\,\text{\normalfont d}\mathbf{x}$.
\end{lem}
\begin{proof}
Let $\sigma(\mathbf{x})=\Sigma\mathbf{x}+\mathbf{b}$, where $\Sigma$ is a linear mapping. Since $T=\sigma(T)$ and $|\sigma(T)|=|\det\Sigma||T|$  we conclude that $|\det\Sigma|=1$ (see e.g. \cite[Th. 5.4.8]{Loj88}).
 Changing variables in integrals we have
 \begin{align*}
	  \int_T f(\mathbf{x})\,\text{\normalfont d}\mathbf{x}
  &=\int_{\sigma(T)} f(\mathbf{x})\,\text{\normalfont d}\mathbf{x}
	=\int_T f(\sigma(\mathbf{x}))|\det\Sigma|\,\text{\normalfont d}\mathbf{x}
	=\int_T f_\sigma(\mathbf{x})\,\text{\normalfont d}\mathbf{x}\,.
 \end{align*}
\end{proof}
\begin{lem}\label{lem:fsigma_is_convex}
If $f:T\to\R$ is convex, then so is $f_\sigma$.
\end{lem}
\begin{proof}
 For $0<t<1$ and $\mathbf{x},\mathbf{y}\in T$ we have
 \begin{multline*}
  f_{\sigma}\bigl(t\mathbf{x}+(1-t)\mathbf{y}\bigr)
  =f\Bigl(\sigma\bigl(t\mathbf{x}+(1-t)\mathbf{y}\bigr)\Bigr)=f\bigl(t\sigma(\mathbf{x})+(1-t)\sigma(\mathbf{y})\bigr)\\
  \xle tf\bigl(\sigma(\mathbf{x})\bigr)+(1-t)f\bigl(\sigma(\mathbf{y})\bigr)
  =tf_{\sigma}(\mathbf{x})+(1-t)f_{\sigma}(\mathbf{y})\,.
 \end{multline*}
\end{proof}

\section{Main result}
Our main result is the following refinement of the classical Hermite--Hadamard inequality:
\begin{thm}\label{th:HH}
 Let $T\subset\R^n$ be a simplex with vertices $\mathbf{v}_0,\dots,\mathbf{v}_{n}$ and let $f:T\to\R$ be a~convex function. Then
 \begin{equation}\label{th:HH:ineq}
  0\xle\mathcal{L}(f,T)\xle n\mathcal{R}(f,T)\,.
 \end{equation}
 The constant $n$ in this inequality cannot be improved.
\end{thm}
\begin{proof}
The inequality $\mathcal{L}(f,T)\xge 0$ follows trivially by~\eqref{eq:HH-simplex}.

Denote by $\mathbf{b}$ the barycenter of $T$ and let $M=\frac{1}{n+1}\bigl(f(\mathbf{v}_0)+\dots+f(\mathbf{v}_n)\bigr)$. Let $\sigma$ be a cyclic permutation of order $n+1$ and $C$ be the subgroup of $S_{n+1}$ generated by $\sigma$.
Define
\begin{equation}\label{eq:F-def}
 F(\mathbf{x})=\frac{1}{n+1}\sum\limits_{\sigma\in C} f_\sigma (\mathbf{x})\,.
\end{equation}
Then
\begin{equation}\label{prop:F_f(b)}
 F(\mathbf{b})=f(\mathbf{b})\,.
\end{equation}
By~\eqref{eq:f_sigma} we get
\begin{equation}\label{prop:F_inv}
 F\bigl(\sigma(\mathbf{x})\bigr)=F(\mathbf{x})\,.
\end{equation}
The formula~\eqref{eq:F-def} gives us
\begin{equation}\label{prop:F_M}
 F(\mathbf{v}_i)=\frac{f(\mathbf{v}_0)+\dots+f(\mathbf{v}_n)}{n+1}=M\,,\quad i=0,\dots,n\,,
\end{equation}
while Lemma \ref{lem:intf=intfsigma} (together with~\eqref{eq:F-def}) yields
\begin{equation}\label{prop:F_int_F}
 \int_T F(\mathbf{x})\,\text{\normalfont d}\mathbf{x}=\int_T f(\mathbf{x})\,\text{\normalfont d}\mathbf{x}\,.
\end{equation}
By Lemma \ref{lem:fsigma_is_convex} the function $F$ is convex. Of course,
\[
 \frac{1}{n+1}\sum_{\sigma\in C}\sigma(\mathbf{x})=\mathbf{b}\,,\quad\mathbf{x}\in T\,,
\]
whence for any $\mathbf{x}=\xi_0\mathbf{v}_0+\dots+\xi_n\mathbf{v}_n\in T$ we have
\begin{align}
 F(\mathbf{b})&=F\Biggl(\frac{1}{n+1}\sum_{\sigma\in C} \sigma(\mathbf{x})\Biggr)
 \xle\frac{1}{n+1}\sum_{\sigma\in C} F\bigl(\sigma(\mathbf{x})\bigr)=F(\mathbf{x})\label{ineq:F(b)<F(x)}\\
 \intertext{and}
 F(\mathbf{x})&=F(\xi_0 \mathbf{v}_0+\dots+\xi_n \mathbf{v}_n)
 \xle \xi_0F(\mathbf{v}_0)+\dots+\xi_n F(\mathbf{v}_n)=M\,.\label{ineq:F(x)<M}
\end{align}
\par\noindent
Consider the function
\[
 G(\mathbf{x})=(n+1)\min_{0\xle j\xle n}\{\xi_j\}F(\mathbf{b}) + \bigl(1-(n+1)\min_{0\xle j\xle n}\{\xi_j\}\bigr)M\,.
\]
Note that $G(\mathbf{b})=F(\mathbf{b})$ and $G\equiv M$ on the boundary of $T$ (since at least one of the barycentric coordinates vanishes here). Our goal is to show that for  all $\mathbf{x}\in T$ the inequality $F(\mathbf{x})\xle G(\mathbf{x})$ holds. To this end fix $\mathbf{x}=\xi_0\mathbf{v}_0+\dots+\xi_n\mathbf{v}_n\in T$ and suppose, that $\xi_i=\min_j\{\xi_j\}$. Then, taking into account that $\mathbf{b}=\frac{1}{n+1}\sum_{j} \mathbf{v}_{j}$ and $\sum_j\xi_j=1$, we can write
\[
 \mathbf{x}=\sum_{j} \xi_j\mathbf{v}_{j}=(n+1)\xi_i \mathbf{b} +\sum_{j\neq i} (\xi_j-\xi_i) \mathbf{v}_j\,.
\]
Clearly all the coefficients are nonnegative and sum up to $1$, so the convexity and~\eqref{prop:F_M} yield
\begin{equation}
 F(\mathbf{x})\xle(n+1)\xi_i F(\mathbf{b}) +\sum_{j\neq i} (\xi_j-\xi_i) F(\mathbf{v}_j)
 =(n+1)\xi_i F(\mathbf{b}) +\bigl(1-(n+1)\xi_i\bigr)M=G(\mathbf{x}).
\label{ineq:F<G}
\end{equation}
We claim that the function $G$ is convex. To check this fix $i\in\{0,\dots,n\}$. Since the barycentric coordinate $\xi_i$ of $\mathbf{x}=\xi_0\mathbf{v}_0+\dots+\xi_n\mathbf{v}_n\in T$ is constant on hyperplanes parallel to the face of~$T$ opposite to the vertex $\mathbf{v}_i$, the function $h_i(\mathbf{x})=\min\Bigl(\frac{1}{n+1},\xi_i\Bigr)$ is concave. Thus the mapping $\mathbf{x}=\xi_0\mathbf{v}_0+\dots+\xi_n\mathbf{v}_n\mapsto\min_i h_i(\mathbf{x}) =\min_j\{\xi_j\}$ is also concave. But $G(\mathbf{x})=M-\bigl(M-F(\mathbf{b})\bigr)\min_j\{\xi_j\}$, hence~$G$ is convex (remember that $F(\mathbf{b})\xle M$). Moreover, the above argument shows that the graph of~$G$ coincides with the lateral surface of the pyramid in $\R^{n+1}$ with base $(T,M)$ and apex $\bigl(\mathbf{b},F(\mathbf{b})\bigr)$ (in the case $n=2$ see Figure~\ref{Pic2}). Therefore
\begin{equation}\label{eq:LG}
 \int_T G(\mathbf{x})\,\text{\normalfont d}\mathbf{x}
 =|T|M-\frac{1}{n+1}|T|\bigl(M-F(\mathbf{b})\bigr)
 =|T|\left(\frac{n}{n+1}M + \frac{1}{n+1}F(\mathbf{b})\right)\,.
\end{equation}
Then integrating \eqref{ineq:F<G} over $T$  we obtain
\[
 \int_T F(\mathbf{x})\,\text{\normalfont d}\mathbf{x}\xle |T|\left(\frac{n}{n+1}M + \frac{1}{n+1}F(\mathbf{b})\right)\,,
\]
which, taking into account \eqref{prop:F_f(b)}--\eqref{prop:F_int_F}, can be written as
\begin{align*}
	\frac{n+1}{|T|}\int_T f(\mathbf{x})\,\text{\normalfont d}\mathbf{x}\xle n\frac{f(\mathbf{v}_0)+\dots+f(\mathbf{v}_n)}{n+1}+f(\mathbf{b})\,,
\end{align*}
or, equivalently $\mathcal{L}(f,T)\xle n\mathcal{R}(f,T)$. Clearly, by~\eqref{eq:LG} we get $\mathcal{L}(G,T)= n\mathcal{R}(G,T)$, so the inequality~\eqref{th:HH:ineq} is optimal.
\end{proof}

\begin{rem}
 As we could see above, the inequality~\eqref{th:HH:ineq} is sharp. For that reason in the multivariate case the left inequality of~\eqref{eq:HH-simplex} does not (in general) estimate the integral mean value better than the right one. The counterexample could be, of course, the exam\-ple given at the end of the above proof.
\end{rem}

\begin{rem}
 Observe that by integrating the inequalities \eqref{ineq:F(b)<F(x)} and \eqref{ineq:F(x)<M} and taking into account the properties~\eqref{prop:F_inv}--\eqref{prop:F_M} we obtain an easy proof of the Hermite--Hadamard inequality~\eqref{eq:HH-simplex}. In the authors' opinion this proof is easier than the proof given in~\cite{Bes08}.
\end{rem}

\section{Fej\'er version}
The classical Fej\'er version of the Hermite--Hadamard inequality (\cf{}~\cite{Fej06}, see also~\cite{DraPea02,MitLac85,PecProTon92}) states that if $f:[a,b]\to\R$ is convex and $g:[a,b]\to\R$ is nonnegative, integrable and symmetric with respect to the midpoint of $[a,b]$, then
\begin{equation*}
 f\left(\frac{a+b}{2}\right)\int_a^b g(x)\,\text{\normalfont d}x
 \xle \int_a^b f(x)g(x)\,\text{\normalfont d}x
 \xle \frac{f(a)+f(b)}{2}\int_a^b g(x)\,\text{\normalfont d}x\,.
\end{equation*}
In this section we reprove analogous result for simplices and provide some estimates for the upper and lower bounds. We start with the multivariate version of Fej\'er's inequality. It was proved, for instance, in~\cite{Neu90}. We provide an alternative proof.
\begin{thm}\label{th:Fejer}
 Let $T$ be a simplex with vertices $\mathbf{v}_0,\dots,\mathbf{v}_{n}$, $f:T\to\R$ be a convex function and $g:T\to\R$ be a nonnegative integrable function such that $g\bigl(\sigma(\mathbf{x})\bigr)=g(\mathbf{x})$ holds for certain cyclic permutation $\sigma\in S_{n+1}$. Then
 \begin{equation*}
  f\Biggl(\frac{1}{n+1}\sum_{i=0}^n\mathbf{v}_i\Biggr)\int_T g(\mathbf{x})\,\text{\normalfont d}\mathbf{x}
  \xle\int_T f(\mathbf{x})g(\mathbf{x})\,\text{\normalfont d}\mathbf{x}
  \xle\Biggl(\frac{1}{n+1}\sum_{i=0}^nf(\mathbf{v}_i)\Biggr)\int_T g(\mathbf{x})\,\text{\normalfont d}\mathbf{x}\,.
 \end{equation*}
\end{thm}
\begin{proof}
 The proof goes exactly along the same line as that of Theorem \ref{th:HH} with two slight differences: we multiply inequalities \eqref{ineq:F(b)<F(x)} and \eqref{ineq:F(x)<M} by $g(\mathbf{x})$ before integrating, and we use the identity
 \[
  \int_T f_\sigma(\mathbf{x})g(\mathbf{x})\,\text{d}\mathbf{x}
 =\int_T f_\sigma(\mathbf{x})g_\sigma(\mathbf{x})\,\text{d}\mathbf{x} =\int_T f(\mathbf{x})g(\mathbf{x})\,\text{d}\mathbf{x}\,.
 \]
\end{proof}
\begin{thm}
 Under hypotheses of Theorem~\ref{th:Fejer}, if
 \begin{align*}
  \mathcal{L}(f,T;g)
   &=\int_T f(\mathbf{x})g(\mathbf{x})\,\text{\normalfont d}\mathbf{x}
    -f\Biggl(\frac{1}{n+1}\sum_{i=0}^n\mathbf{v}_i\Biggr)\int_T g(\mathbf{x})\,\text{\normalfont d}\mathbf{x}\,,\\[1ex]
  \mathcal{R}(f,T;g)
  &=\Biggl(\frac{1}{n+1}\sum_{i=0}^nf(\mathbf{v}_i)\Biggr)\int_T g(\mathbf{x})\,\text{\normalfont d}\mathbf{x}
   -\int_T f(\mathbf{x})g(\mathbf{x})\,\text{\normalfont d}\mathbf{x}\,,\\[1ex]
  \Delta(f,T)&=\frac{1}{n+1}\sum_{i=0}^n f\mathbf{v}_i)-f\Biggl(\frac{1}{n+1}\sum_{i=0}^n\mathbf{v}_i\Biggr)\,,
 \end{align*}
 then
 \begin{align}
  0&\xle\Delta(f,T)\int_T g(\mathbf{x})\cdot(n+1)\min_j\{\xi_j\}\,\text{\normalfont d}\mathbf{x}
  \xle\mathcal{R}(f,t;g)\,,\label{ineq:HHF_R}\\[1ex]
  0&\xle\mathcal{L}(f,t;g)
  \xle\Delta(f,T)\int_T g(\mathbf{x})\cdot\bigl(1-(n+1)\min_j\{\xi_j\}\bigr)\,\text{\normalfont d}\mathbf{x}\,.\label{ineq:HHF_L}
 \end{align}
\end{thm}
\begin{proof}
Let $\alpha=\int_T g(\mathbf{x})\cdot(n+1)\min_j\{\xi_j\}\,\text{\normalfont d}\mathbf{x}$. Multiplying \eqref{ineq:F<G} by $g$ and integrating we obtain
\[
 \int_T f(\mathbf{x})g(\mathbf{x})\,\text{d}\mathbf{x}\xle f(\mathbf{b})\alpha+M\int_T g(\mathbf{x})\,\text{d}\mathbf{x} -M\alpha
\]
which immediately yields \eqref{ineq:HHF_R}. Subtracting $f(\mathbf{b})\int_T g(\mathbf{x})\,\text{d}\mathbf{x}$ from both sides we obtain \eqref{ineq:HHF_L} (the inequality $\mathcal{L}(f,t;g)\xge 0$ follows trivially by Theorem~\ref{th:Fejer}).
\end{proof}

\begin{rem}
 The full counterpart of Theorem~\ref{th:HH:ineq} cannot be obtained. Namely, the inequa\-lity $\mathcal{L}(f,T;g)\xle n\mathcal{R}(f,T;g)$ does not hold for the arbitrarily chosen positive symmetric weight function $g$ and a~convex function~$f$.  A counterexample is quite easy to show. It is enough to take $n=1$, $T=[-1,1]$, $f(x)=g(x)=x^2$. In these settings we have
 \begin{equation*}
  \mathcal{L}(f,T;g)=\int_{-1}^1x^4\,\text{d}x=\frac{2}{5}\,,\qquad
  \mathcal{R}(f,T;g)=\int_{-1}^1x^2\,\text{d}x-\int_{-1}^1x^4\,\text{d}x=\frac{4}{15}\,.
 \end{equation*}
 Another example could be given for any simplex $T\subset\R^n$. We will show that for the arbitrarily chosen constant $N>0$ it is possible to take a~convex function $F:T\to\R$ and a~nonnegative, integrable and symmetric function $g:T\to\R$ such that $\mathcal{L}(F,T;g)>N\mathcal{R}(F,T;g)$. We only sketch the construction.
 \par\smallskip
 Let us build the convex function $F$ similarly as in the proof of Theorem~\ref{th:HH}, with the additional conditions $F(\mathbf{b})=0$, $F(\mathbf{v}_i)=1$, $i=0,\dots,n$. Then for $\mathbf{x}=\xi_0\mathbf{v}_0+\dots+\xi_n\mathbf{v}_n\in T$ we have $G(\mathbf{x})=1-(n+1)\min_j\{\xi_j\}$. Next, for $a\in(0,1)$ take $g(\mathbf{x})=\alpha\max\bigl\{G(\mathbf{x})-a,0\bigr\}$ with such factor $\alpha$, that $\int_Tg(\mathbf{x})\,\text{\normalfont d}\mathbf{x}=1$. Then $\int_TF(\mathbf{x})g(\mathbf{x})\,\text{\normalfont d}\mathbf{x}\xrightarrow[a\to 1]{}1$, so $\mathcal{L}(F,T;g)\xrightarrow[a\to 1]{}1$ and $\mathcal{R}(F,T;g)\xrightarrow[a\to 1]{}0$.
 \par\smallskip
 If $n=1$ and $T=[-1,1]$, then the above construction leads to
 \[
  F(x)=G(x)=|x|\,,\quad g(x)=\frac{\max\bigl\{|x|-a,0\bigr\}}{(1-a)^2}
 \]
 and we arrive at
 \[
  \int_{-1}^1 F(x)g(x)\,\text{\normalfont d}x=\frac{a+2}{3}\xrightarrow[a\to 1]{}1\,,
  \quad\int_{-1}^1 g(x)\,\text{\normalfont d}x=1\,.
 \]
\end{rem}
\begin{rem}
After submission of this paper it came to our attention that similar result was obtained by Mitroi and Spiridon, see the paper \cite{MitSpi11}.
\end{rem}

\end{document}